\documentclass[12pt]{amsart}

\usepackage{latexsym}
\usepackage{amsfonts}
\usepackage{amsmath}
\usepackage{amsthm}
\usepackage{amssymb}
\usepackage{amscd}

\usepackage[english]{babel}
\newcommand{\hot}{\hat{\otimes}}

\newcommand{\CC}{\mathbb{C}}
\newcommand{\RR}{\mathbb{R}}
\newcommand{\ZZ}{\mathbb{Z}}
\newcommand{\distr}{\mathcal{D}}
\newcommand{\smallGL}{GL_{n-1}(\RR)}
\newcommand{\GLtransp}{\widetilde{GL_{n-1}(\RR)}}
\newcommand{\raw}{\rightarrow}
\newcommand{\xx}{\mathfrak{X}}

\newtheorem{thm}{Theorem}[section]
\newtheorem{thm*}{Theorem}
\newtheorem{lem}[thm]{Lemma}
\newtheorem{lem*}[thm*]{Lemma}
\newtheorem{prop}[thm]{Proposition}
\newtheorem{cor}[thm]{Corollary}
\newtheorem{prop*}[thm*]{Proposition}
\newtheorem{cor*}[thm*]{Corollary}

\newtheorem{defn*}[thm*]{Definition}
\newtheorem{thm-defn}[thm]{Theorem-Definition}

\theoremstyle{remark}

\begin{document}
\author{Alexander Kemarsky}
\address{Mathematics Department, Technion - Israel Institute of Technology, Haifa, 32000 Israel}
\email{alexkem@tx.technion.ac.il}

\title{Distinguished representations of $GL(n,\CC)$}

\begin{abstract}
Let $V$ be a $GL_n(\RR)$-distinguished, irreducible, admissible representation of $GL_n(\CC)$.
 We prove that any continuous linear functional on $V$, which is invariant
under the action of the real mirabolic subgroup, is automatically $GL_n(\RR)$-invariant.
\end{abstract}

\date{\today}

\maketitle
\makeatletter
\@setabstract

\numberwithin{equation}{section}

\tableofcontents

\begin{section}{Introduction}
Let $G = GL(n,\CC)$, $H = GL(n,\RR)$ and $P \le H$ be the standard mirabolic
subgroup, i.e., the group consisting of matrices with last row equals to
$(0,0,...0,1)$. We identify $H \backslash G$ with the space of matrices
$$X =  \bigg\{ x \in G  \big| x \cdot \bar{x} = I_n \bigg\},$$
via the isomorphism $Hg \mapsto \bar{g}^{-1} \cdot g $ (for the proof of the surjectivity of
the map see Lemma \ref{matrix Hilbert 90}).
The group $G$ acts on $X$ by the twisted conjugation,
$x \rho(g) := \bar{g}^{-1}xg$.\newline
For a topological vector space $V$, we denote by $V^*$ the topological dual of $V$, i.e.,
the space of all continuous maps from $V$ to $\CC$.
Let
$$\distr(X) = C_c^{\infty}(X)^*$$ be the space of distributions on $X$.
The space $\distr(X)$ is a topological space with the standard topology
($T_n \rightarrow T$ if $T_n(f) \raw T(f)$ for every $f \in C_c^{\infty}(X)$.) \newline
In this paper we work with the category of the admissible smooth Fr\'echet representations of moderate growth,
 see \cite[Section 11.5]{WallachB2}, see also \cite[Section 2.1]{AGS}. 
 \newline
Let $(\pi,V)$ be a representation of $G$. The representation $(\pi,V)$ is called $H$-distinguished if there exists
a non-zero continuous linear map
$L:V \rightarrow \CC $, such that $$L(\pi(h)v) = L(v) \hspace{3mm} \forall v \in V, h \in H.$$
We denote the space of all such linear maps by $\big(V^* \big)^H.$ \newline
Let $V,W$ be two representations in our category. Then we denote by $V \hot W$ the completed tensor product with
the projective topology (this is the $\pi$-topology in \cite[Definitions 43.2 and 43.5]{Treves}). \newline
We prove the following theorem
\begin{thm}\label{main thm}
Let $(\pi,V)$ be an irreducible, admissible $H$-distinguished representation of $G$. Then
$$\big( V^* \big)^P = \big(V^* \big)^H. $$
\end{thm}
The proof is based on section 3 of the article of Offen \cite{O}, see also \cite{Ok}. 
It is splitted into two steps as follows. \newline
Step 1: Reduction of the problem to a question on a single symmetric space $X$, defined above. More specificially,
we prove that it is enough to show that $\distr(X)^P = \distr(X)^H $. \newline
Step 2: We prove
\begin{prop}\label{main prop}
$\distr(X)^P = \distr(X)^H $.
 \end{prop}
 Let $\xi \in \distr(X)^P$. Note that the transpose acts on $X$. Indeed, let $x \in X$. Then $x \bar{x} = I_n$.
 Let us take the transpose and bar on both sides of the last equality. Then we obtain $(^t x )(^t\bar{x})= I_n $.
 Since the image of $P$ under the map $g \mapsto ^t \bar{g}^{-1}$ is $^tP$ and since
  $^t(x\rho(g)) =^tx \rho(^t \bar{g}^{-1}), $
we obtain that $\xi \in \distr(X)^{^tP} $. \newline
 In order to prove Proposition \ref{main prop} we use an idea that was suggested by Gourevitch (in personal
 communication).
 Actually, we prove a more general result, namely the following
 \begin{prop} \label{invariance to transpose}
 Let $\xi \in \distr(X)^{\smallGL}$. Then $^t \xi = \xi$, i.e., $\xi$ is invariant under the transposition.
 \end{prop}
 Note that Proposition \ref{invariance to transpose} implies Proposition \ref{main prop}.
 Indeed, let $\xi \in \distr(X)^P$. Since $\smallGL < P$ we have $\xi \in \distr(X)^{\smallGL}$.
 Thus, by Proposition \ref{invariance to transpose},
 the distribution $\xi$ is invariant under the transposition, and
  hence $\xi \in \distr(X)^{^tP} $. But $H$ is generated by groups $P$ and $^tP$. Hence $\xi \in \distr(X)^H$.
\newline
 This result is important since it allows us  to prove that the following two linear forms on the
Whittaker model of an irreducible, generic, and unitarizable distinguished representation $(\pi,V)$ defined by
$$W \mapsto \int_{U_n(\RR) \backslash P_n(\RR)}W(p)dp,  $$ and
$$W \mapsto \int_{U_n(\RR) \backslash P_n(\RR)}W\left( \left(
                                                   \begin{array}{cc}
                                                     0 & 1 \\
                                                     I_{n-1} & 0 \\
                                                   \end{array}
                                                 \right)
p \right)dp $$
are $GL_n(\RR)$-invariant. Here, $U_n(\RR)$ is the group of all upper triangular unipotent matrices of size $n \times n$.
\newline
The last result, combined with certain
multiplicity one theorems, allows us to recover properties of epsilon factors of distinguished representations.
\subsection*{Acknowledgements}
I would like to thank Prof. Omer Offen for posing me this question and providing many explanations
of the subject. \newline
I'm grateful to Prof. Dmitry Gourevitch for many fruitful discussions and
his help.
\newline
I also wish to thank Prof. Moshe Baruch, Dror Speiser and Tal Horesh for useful remarks.
\end{section}
\begin{section}{Notation and preliminaries}
We begin with the following simple lemma, which states the surjectivity of the map
$Hg \mapsto \bar{g}^{-1} \cdot g $.
\begin{lem}\label{matrix Hilbert 90}
Let $g \in G$ be a matrix such that $g \cdot \bar{g} = I_n$, where $I_n$ is the identity
matrix of size $n \times n$. Then there exists a matrix $h \in G$ such that
$g = \bar{h}^{-1} \cdot h$.
\end{lem}
\begin{proof}
Let $\lambda \in \CC$ be a scalar, define $h = \bar{\lambda} I + \lambda g$.
One easily checks that $$\bar{h} \cdot g = \lambda g + \bar{\lambda} \cdot I = h.$$
Hence, if $h$ is invertible we get $g = \bar{h}^{-1} \cdot h$. Since $\lambda$ is arbitrary, we can choose it in a way, so that
$-\frac{\bar{\lambda}}{\lambda}$ is not an eigenvalue of $g$. This
finishes the proof.
\end{proof}
Consider an action of the $2$-element group $\ZZ_2$ on $H$ given by:
$h \mapsto {}^th^{-1} $. It induces the semidirect product
$$\widetilde{H} := H \rtimes \ZZ_2. $$
Define a character $\chi$ on $\tilde{H}$ by $\chi(g,n):=(-1)^n$. \newline
We view the group $GL_{n-1}(\RR)$ as a subgroup of $H$ - i.e. matrices of the form
\[
\left(
  \begin{array}{cc}
    A_{(n-1)\times(n-1)} & 0 \\
    0 & 1 \\
  \end{array}
\right).
\]
Similarly, define $\GLtransp := GL_{n-1}(\RR) \rtimes \ZZ_2$.
Clearly, this is a subgroup of $\widetilde{H}$. \newline
In the first stage of the proof we will use the following result (\cite[Theorem A.0.3]{AGS} )
\begin{thm} (\textbf{Frobenius descent})
Let $G$ be a Lie group, assume $G$ acts on two smooth manifolds $X$ and $Y$, where the action
on $Y$ is transitive.
Denote the actions by $\rho_1$ and $\rho_2$ respectively.
Let $p:X \twoheadrightarrow Y$ be
a smooth surjective map that commutes with the action of $G$, i.e.,
for all $g \in G$ the following diagram commutes:
\[
\begin{CD}
X @> \rho_1(g) >> X\\
@VV p V @VV p V\\
Y @> \rho_2(g) >> Y
\end{CD}
.
\]
For every $y \in Y$, let $G_y := Stab_G(y)$, the stabilizer of $y$ in $G$.
Then $$\distr(X)^G = \distr(p^{-1}(y))^{G_y}.$$
\end{thm}
Let $(\pi,V)$ be a continuous representation of the Lie group $G$ on a Fr\'echet space $V$.
Define the space $V^{(1)}$ of differentiable vectors in $V$ to be the set of all $v \in V$ such that
the derivative $$ \frac{d}{dt}\pi(\exp(tx))v \big|_{t=0}$$
exists for all $x \in \text{Lie}(G)$, where $\text{Lie}(G)$ is the Lie algebra of $G$.
The resulting vector is denoted by $\pi(x)v$. \newline
We now define inductively for $n\in \mathbb{N}$
$$V^{(n)}= \{ v \in
V^{(n-1)}
 \big| \pi(x)v \in V^{(n-1)}  \text{ for all } x \in \text{Lie}(G))   \} .$$
Define $$V^{\infty} = \bigcap_{n=1}^{\infty} V^{(n)}.$$
Let $C_c^{\infty}(G)$ be the space of smooth compactly supported functions on $G$.\newline
The following two theorems will be used in lemma \ref{Bernstein's lemma}.
\begin{thm} \label{Dixmier-Maliavin} (\cite{Cas}, \textbf{Dixmier-Maliavin})
Let $G$ be a Lie group and $(\pi,V)$ a continuous representation of $G$
on a Fr\'echet space. Then every $v \in V^{\infty}$ may be represented as a finite
linear combination $$v = \sum \pi(f_k)v_k, $$ where $f_k \in C_c^{\infty}(G)$
and $v_k \in V^\infty$, for all $k$.
\end{thm}
\begin{thm} (\cite[Thm 51.6,(c)]{Treves})
Let $K \subseteq \RR^m$ and $L \subseteq \RR^n$ be compact sets. Then
$$C_c^{\infty}(K) \hot C_c^{\infty}(L) \cong C_c^{\infty}(K \times L). $$
\end{thm}
\begin{cor}\label{2 variables tensor inc}
Let $G = GL_n(\CC) \subseteq \RR^{2n^2}$. Then
$$C_c^{\infty}(G \times G) \subseteq C_c^{\infty}(G) \hot C_c^{\infty}(G).$$
\end{cor}
\begin{proof}
Let $f \in C_c^{\infty}(G \times G)$. Then $supp(f)$ is a compact set in $G \times G$, namely,
there exist compact sets $K \subseteq \RR^{2n^2}$ and $L \subseteq \RR^{2n^2}$, such that
$supp(f) \subseteq K \times L$.
Hence, $$f \in  C_c^{\infty}(K \times L) \cong C_c^{\infty}(K) \hot C_c^{\infty}(L).$$
But $C_c^{\infty}(K) \hot C_c^{\infty}(L) \subseteq C_c^{\infty}(G) \hot C_c^{\infty}(G)$ and we obtain
that $f \in C_c^{\infty}(G) \hot C_c^{\infty}(G) $.
\end{proof}
Let $M$ be a smooth manifold. Denote by $S(M)$ the Fr\'echet space of Schwartz functions
on $M$. Denote by $$S^*(M):= \big( S(M) \big)^*$$ the space of Schwartz distributions on $M$. \newline
In the second part of the proof Theorem \ref{main thm}, we use the following theorem of
Aizenbud and Gourevitch, see further \cite[Propositions 3.1.3, 3.1.4 and Theorem A]{AG}, see also \cite{SunZhu}.
\begin{thm}\label{Ais-Gur thm}
Any $GL_n(\RR)$-invariant distribution on $M_{n+1}(\RR)$ is invariant
under the transposition.
\end{thm}
In fact, Theorem \ref{Ais-Gur thm} implies that $\distr(M_{n+1}(\RR))^{\widetilde{GL_{n}(\RR)},\chi} = 0$. 
Proposition 3.1.5 in \cite{AG} together with Theorem 3.1.1 imply that
$$S^{*}(M_{n+1}(\RR))^{\widetilde{GL_{n}(\RR )},\chi} = 0.$$
It follows that a $GL_n(\RR)$-invariant Schwartz distribution on $M_{n+1}(\RR)$
is also invariant under the transposition.
Then the implication
$$\big( S^{*}(M_{n+1}(\RR))^{\widetilde{GL_{n}(\RR)},\chi} = 0 \big)  \Rightarrow
\big( \distr(M_{n+1}(\RR))^{\widetilde{GL_{n}(\RR)},\chi} = 0 \big),$$
which follows from  \cite[Theorem 4.0.2]{AG2}, implies that a 
$GL_n(\RR)$-invariant distribution on $M_{n+1}(\RR)$
is also invariant under the transposition.
\end{section}
\begin{section}{Proposition \ref{main prop} implies theorem \ref{main thm}}
\begin{lem}\label{Bernstein's lemma}
Let $(\pi,V)$ be an irreducible, admissible representation of $G$. Let $(\tilde{\pi},\tilde{V})$ be
the contragredient of $(\pi,V)$. Then there exists a morphism
of $G \times G$-modules with a dense image
$$A_{\pi}: C_c^{\infty}(G) \rightarrow V \hat{\otimes} \tilde{V}. $$
\end{lem}
\begin{cor}
The dual map $$A_{\pi}^* : V^{*} \hat\otimes \tilde{V}^* \raw \distr(G)$$
is an injective morphism of $G \times G$-modules.
\end{cor}
\begin{proof}
The proof follows \cite[pp. 76-77]{Ber}.
We define the map $A_{\pi}$ as a composition
$$A_{\pi}:C_c^{\infty}(G) \raw Hom(V,V) \simeq V \hat\otimes V^{*}, $$
where the isomorphism $Hom(V,V) \simeq V \hot V^*$ follows from the standard theory of nuclear Fr\'echet spaces,
see, for example, \cite[Lemma A.0.8]{AG}. \newline
The map $C_c^{\infty}(G) \raw Hom(V,V)$ is given by
$$f \mapsto \big(v \mapsto \pi(f)v := \int_{G} f(g) (\pi(g)v) dg \big).$$
Since $f \in C_c^{\infty}(G)$, it follows that the image of $A_{\pi}$ is
contained in $(V \hot V^*)^{\infty}$. By the Dixmier-Malliavin theorem (Theorem \ref{Dixmier-Maliavin}), the space of smooth
vectors in $V \hot V^*$ is a subspace of $V \hot \tilde{V} $.
Indeed, let $v \in (V \hot V^*)^{\infty}$. Then by the Dixmier-Malliavin theorem,
$$v = \sum_{i=1}^{k} \pi(F_i)v_i, $$
where $v_i \in V \hot V^*$ and $F_i \in C_c^{\infty}(G \times G)$.
Since $C_c^{\infty}(G \times G)(V \hot V^*) \subseteq V \hot \tilde{V}$, we have that
$v \in V \hot \tilde{V}$. \newline
To prove the inclusion $C_c^{\infty}(G \times G)(V \hot V^*) \subseteq V \hot \tilde{V}$ one can
argue as follows.
We know that for $f \in C_c^{\infty}(G),v \in V$ and $v' \in V^*$ one has that $\pi(f)v \in V$ and
$\pi(f)v' \in \tilde{V}$. By taking a tensor product of these maps
we obtain a map $$C_c^{\infty}(G) \hot V \hot C_c^{\infty}(G) \hot V^* \mapsto V \hot \tilde{V}.$$
We then have a map
\begin{equation}\label{smooth G x G}
\big( C_c^{\infty}(G) \hot C_c^{\infty}(G) \big) \hot \big( V \hot V^* \big) \mapsto V \hot \tilde{V}.
\end{equation}
By Corollary \ref{2 variables tensor inc}, we have the following inclusion:
\begin{equation}\label{2 variables tensor}
 C_c^{\infty}(G \times G) \subseteq C_c^{\infty}(G) \hot C_c^{\infty}(G) .
\end{equation}
Combining \ref{smooth G x G} and \ref{2 variables tensor}, we obtain
a non-zero morphism of $G \times G$-modules
$$A_{\pi}: C_c^{\infty}(G) \raw V \hot \tilde{V}. $$
Since $V \hot \tilde{V}$ is an irreducible $G \times G$-module - see \cite[p.289]{AG} - 
 we obtain that the map $A_\pi$ has a dense image. \newline \newline
\end{proof}

\begin{thm}\label{distributions on G}
Suppose that $$\distr(G)^{H \times P} = \distr(G)^{H \times H}. $$
Then Theorem \ref{main thm} holds.
\end{thm}
\begin{proof}
The proof repeats almost verbatim \cite[p. 177]{O}.\newline
Let $(\pi,V)$ be an irreducible, admissible $H$-distinguished representation. Then the contragredient representation
$(\tilde{\pi},\tilde{V}) $ is also an irreducible, admissible $H$-distinguished representation. Indeed, by the theorem of
Aizenbud, Gourevitch and Sayag, \cite[Theorem 2.4.2]{AGS}, 
see also the theorem of Gelfand-Kazhdan in the $p$-adic case, \cite{GK}, the contragredient representation
is $H'$-distinguished, where $$H' = \{^th^{-1} \big| h \in H \}.$$
Since clearly $H' = H$, we get that the representation $(\tilde{\pi},\tilde{V}) $
is $H$-distinguished. \newline
Take two non-zero linear forms
$\lambda \in \big( \tilde{V}^* \big)^H$ and $\mu \in \big( V^*) ^P$.
Then, by Corollary \ref{2 variables tensor}, $$0 \neq A_{\pi}^*(\mu \otimes \lambda) \in \distr(G)^{P \times H}.$$
By the assumption, $A_{\pi}^*(\mu \otimes \lambda) \in \distr(G)^{H \times H}.$ Since $A_{\pi}^*$ is an injective
morphism of $G \times G$-modules, it follows that $\mu \otimes \lambda \in (V^* \hat{\otimes} \tilde{V}^*)^{H \times H}$.
Therefore $\mu \in (V^*)^H$.
\end{proof}
We see that in order to prove Theorem \ref{main thm} it is enough to prove that
 $$\distr(G)^{H \times P} = \distr(G)^{H \times H}.$$
We want to show that
\begin{equation}\label{first Frob} \distr(G)^{H \times P} \simeq \distr(H \backslash G)^{P} \simeq \distr(X)^{P},
\end{equation}
\begin{equation}\label{second Frob} \distr(G)^{H \times H} \simeq \distr(H \backslash G)^H \simeq \distr(X)^{H}.
\end{equation}
These isomorphisms follow from the next lemma.
\begin{lem}\label{Frobenius descent}
Let $Q$ be a Lie group and let $R$ be a closed Lie subgroup of $Q$. Then for any
subgroup $Q'<Q$ we have
\begin{equation}
\distr(R \backslash Q)^{Q'} \simeq \distr(Q)^{R \times Q'}.
\end{equation}
\end{lem}
\begin{proof}
We prove this by applying the Frobenius descent twice. Let $$Y = R \backslash Q \times Q.$$
The group $Q \times Q'$ acts on $Y$ by $$\rho_1(q,q') \cdot (x,x') := (xq^{-1},qx'q'^{-1}).$$
Take $Z = R \backslash Q$ with $Q \times Q'$ acting on $Z$ by $$\rho_2(q,q') \cdot x := xq^{-1}. $$
This action is transitive and we have a commutative diagram
\begin{equation}
\begin{CD}
Y @> \rho_1 >> Y\\
@VV \phi V @VV \phi V\\
Z @> \rho_2 >> Z
\end{CD}
,
\end{equation}
where $\phi(x,q') = x$ is the projection onto the first coordinate. \newline
Let $x = Re_Q \in Z$ be the identity coset.
The fiber above $Re_Q$ is $$\phi^{-1}(x) = \big\{x \big\} \times Q \simeq Q$$ and
the stabilizer is $(Q \times Q')_x = R \times Q'$. We thus obtain
\begin{equation}\label{first frobenius}
\distr(Y)^{Q \times Q'} \simeq \distr(Q)^{R \times Q'}.
\end{equation}
The group $Q \times Q'$ acts on the group $W=Q$ by $$\rho_3(q,q') \cdot x := qxq'^{-1}. $$
This action is transitive and we have a commutative diagram
\begin{equation}
\begin{CD}
Y @> \rho_1 >> Y\\
@VV \psi V @VV \psi V\\
W @> \rho_3 >> W
\end{CD}
,
\end{equation}
where $\psi(x,q') = q'$ is the projection onto the second coordinate. \newline
The stabilizer of $x=e_Q \in W$ is $(Q \times Q')_x = \big\{(q,q') \in Q \times Q' \big| q=q' \big\} \simeq Q'$,
and the fiber above $x$ is $$\psi^{-1}(x) = R \backslash Q \times x \simeq R \backslash Q.$$
We have that
\begin{equation}\label{second frobenius}
\distr(Y)^{Q \times Q'} \simeq \distr(R\backslash Q)^{Q'}.
\end{equation}
Combining (\ref{first frobenius}) and (\ref{second frobenius}), the result follows.
\end{proof}
Note that $H$ is indeed a closed subgroup of $G$. Thus, isomorphims in (\ref{first Frob}) follow from Lemma
\ref{Frobenius descent} with
$Q = G, R = H,Q' = P $ and isomorphims (\ref{second Frob}) follow from this lemma with
$Q=G, R = H, Q' = H$.
To summarise, the main theorem follows by implying (\ref{first Frob}) and (\ref{second Frob})
to Proposition \ref{main prop}.
\end{section}
\begin{section}{Proof of the main proposition \ref{invariance to transpose}}
Before we begin the proof we shall introduce some notation.\newline
Let $$\xx = \bigg\{ x \in M_n(\CC): x+\bar{x} = 0 \bigg\} = i\cdot M_n(\RR). $$
Thus, $x \mapsto ix$ is an isomorphism $M_n(\RR) \simeq \xx $ of $H$-spaces, where $H$ acts on both
$M_n(\RR)$ and $\xx $
by conjugation. \newline
Let $\theta: \xx \rightarrow \RR$, $$\theta(x) = det(I_n+x)det(I_n-x) = det(I_n+x)\overline{det(I_n+x)}.$$
Clearly, $\theta$ is a continuous $H$-invariant map. The proof of Proposition \ref{invariance to transpose}
is given in two steps.
We begin with the following linearization argument (the idea appeared in \cite[p. 1511]{AGS}).
\begin{lem}
Let $$\xx_0 = \theta^{-1}(\RR^*) = \bigg\{ x \in \xx \big| \theta(x) \neq 0 \bigg\}.$$
Then $\smallGL$ acts on $\xx_0$, and any $\xi \in \distr(\xx_0)^{\smallGL} $ is invariant
under the transposition.
\end{lem}
\begin{proof}
The proof is an adaptation of Proposition 3.2.3 from [AGS] 
to our case. The idea is to
take a distribution on $\xx_0$ and "extend" it to a distribution on $\xx$. By Theorem
 \ref{Ais-Gur thm} we have
 that any $\xi \in \distr(\xx)^{\smallGL} $ is invariant under the transposition.
Hence, we have $\distr(\xx)^{\GLtransp,\chi} = 0. $ We need to prove that
$\distr(\xx_0)^{\GLtransp,\chi} = 0. $ \newline
Assume the contrary, and let $ 0 \neq \xi \in \distr(\xx_0)^{\GLtransp,\chi}$. Take $p \in \text{Supp}(\xi)$
and let $t = \theta(p)$. Note that $t \neq 0$. Let $f \in \mathcal{S}(\RR)$ be such that $f$ vanishes in a
neighborhood of zero and such that $f(t) \neq 0$. Consider
$\xi' := (f \circ \theta) \cdot \xi, $
i.e., $$\xi'(h) = \xi(h \cdot (f \circ \theta))$$
for all $h \in C_c^{\infty}(\xx_0)$.
Then $\xi' \in \distr(\xx_0)^{\GLtransp,\chi}$. However, we can extend $\xi'$ "by zero" to
$$\xi'' \in \distr(\xx)^{\GLtransp,\chi} = 0.$$ Therefore $\xi' = 0$, and we obtained a contradiction!
\end{proof}
Next, define a covering of $X$ by open $H$-subspaces as follows. Note that if
$\lambda \in \CC^*$ is an eigenvalue of $x \in X$ then $\lambda \overline{\lambda} = 1$.
For every such $\lambda$ let
$$X_{\lambda} =  \bigg\{ x \in X \big| det(x-\lambda I_n) \neq 0 \bigg\}.$$
Note that $X - X_{\lambda}$ consists of all the matrices in $X$ with $\lambda$ as an eigenvalue.
Since every matrix $x \in X$ has at most $n$ different eigenvalues we obtain a finite covering
of $X$ by open $H$-subspaces
$$X = \bigcup_{i=1}^{n+1} X_{\lambda_i}, $$
where $\lambda_1, ..., \lambda_{n+1}$ are any $n+1$ distinct chosen numbers such that $\lambda_j
\overline{\lambda_j} = 1$.
\newline Next, we have an isomorphism of $H$-spaces $\xx_0$ and $X_{\lambda}$.
The isomorphism is given by a Cayley transform  $$\xi_{\lambda}(x) = (x+\lambda I_n)(x-\lambda I_n)^{-1}. $$
The proof that this is indeed an isomorphism is given in [O], pp. 181-182.
Our purpose is to prove that  every distribution $\xi \in \distr(X)$ which is $\smallGL$-invariant is
invariant under the transposition. This is equivalent to
\begin{equation}\label{distr equality}
\distr(X)^{\GLtransp,\chi} = 0 .
\end{equation}
The last equation now follows from the next simple lemma, which is a special case of \cite[Proposition 2.2.6]{AG}.
\begin{lem}
Assume that a Lie group $G$ acts on a smooth manifold $X$. Let
$$X= \bigcup_{i=1}^{k} A_i $$
be a finite open covering, where $A_i$ are $G$-invariant open subsets. Suppose $\distr(A_i)^{G,\chi} = 0$ for all $i=1,...,k$.
Then $$\distr(X)^{G,\chi} = 0.$$
\end{lem}
This completes the proof of Proposition \ref{main prop} and therefore the proof of Theorem \ref{main thm}.
\end{section}

\end{document}